\documentclass[gdplain]{geradwp}

\PassOptionsToPackage{hyphens}{url}

\usepackage[ruled]{algorithm}
\usepackage[noend]{algpseudocode}
\usepackage{hyperref}

\GDcoverpagewhitespace{6.8cm}
\graphicspath{{Figures/}} % graphicx pkg setup
\hypersetup{colorlinks,%
citecolor={black}, % Change for "black" with natbib
urlcolor={blue},
linkcolor={blue},
breaklinks={true}
}

\makeatletter
\ifthenelse{\isundefined{\ALG@name}}{}%
{%
\renewcommand{\ALG@name}{\sffamily\footnotesize Algorithm}
}
\makeatother
\ifthenelse{\isundefined{\SetAlCapNameFnt}}{}%
{% 
\SetAlCapNameFnt{\footnotesize}
\SetAlCapFnt{\sffamily\footnotesize}
}

\usepackage[square,numbers]{natbib}
\usepackage{stmaryrd}
\usepackage{xspace}
\usepackage[textsize=scriptsize]{todonotes}

\newtheorem{proposition}{Proposition}
\newtheorem{lemma}{Lemma}
\newtheorem{theorem}{Theorem}

\newtheorem{assumption}{Assumption}

\def\ddsm{dDSM\xspace}
\def\mads{MADS\xspace}
\def\poll{\textbf{Poll}\xspace}
\def\revpoll{\textbf{Revealing Poll}\xspace}
\def\search{\textbf{Search}\xspace}
\DeclareMathOperator*{\minimize}{minimize}
\DeclareMathOperator*{\argmin}{argmin}

\def\R{\mathbb{R}}
\def\N{\mathbb{N}}
\def\Z{\mathbb{Z}}
\def\S{\mathbb{S}}
\def\B{\mathcal{B}}
\def\P{\mathcal{P}}
\def\D{\mathcal{D}}
\def\ll{\llbracket}
\def\rr{\rrbracket}
\newcommand{\norm}[1]{\left\lVert#1\right\rVert}
\newcommand{\abs}[1]{\left\lvert{#1}\right\rvert}
\def\int{\mathrm{int}}
\def\cl{\mathrm{cl}}
\GDtitle{Erratum, counterexample and an additional revealing poll step for a result of ``Analysis of direct searches for discontinuous functions''}
\GDmonth{Novembre}{November}
\GDyear{2022}
\GDnumber{XX}
\GDauthorsShort{C. Audet, P.-Y. Bouchet, L. Bourdin}
\GDauthorsCopyright{P.-Y. Bouchet}
\GDpostpubcitation{Hamel, Benoit, Karine H\'ebert (2022). ``Un exemple de citation'', \emph{Journal of Journals}, vol. X issue Y, p. n-m}{https://www.gerad.ca/fr}
\GDsupplementname{Internet Appendix}
\GDrevised{Novembre}{November}{2022}

\begin{document} 

\GDcoverpage

\begin{GDtitlepage}

\begin{GDauthlist}
	\GDauthitem{C. Audet \ref{affil:gerad}}
	\GDauthitem{P.-Y. Bouchet \ref{affil:gerad}}
	\GDauthitem{L. Bourdin \ref{affil:unilim}}
\end{GDauthlist}

\begin{GDaffillist}
	\GDaffilitem{affil:gerad}{GERAD, Montr\'eal (Qc), Canada, H3T 1J4}
	\GDaffilitem{affil:unilim}{XLIM Research Institute, UMR CNRS 7252, University of Limoges, France}
\end{GDaffillist}

\begin{GDemaillist}
	\GDemailitem{charles.audet@gerad.ca}
	\GDemailitem{pierre-yves.bouchet@polymtl.ca}
	\GDemailitem{loic.bourdin@unilim.fr}
\end{GDemaillist}

\end{GDtitlepage}

%% %%%%%%%%%%%%%%%%%%%%%%%%%%%%%%%%%%%%%%%%%%%%%%%%%%%%%%%
%% %%%%%%%%% Résumés, mots-clés, remerciements %%%%%%%%%%%
%% %%%%%%% Abstract, keywords, acknowledgements %%%%%%%%%%
%% %%%%%%%%%%%%%%%%%%%%%%%%%%%%%%%%%%%%%%%%%%%%%%%%%%%%%%%

\GDabstracts

\begin{GDabstract}{Abstract}
This note provides a counterexample to a theorem announced in the last part of the paper \textit{Analysis of direct searches for discontinuous functions}, Mathematical Programming Vol. 133, pp.~299--325, 2012.
The counterexample involves an objective function $f: \R \to \R$ which satisfies all the assumptions required by the theorem but contradicts some of its conclusions.
A corollary of this theorem is also affected by this counterexample.
The main flaw revealed by the counterexample is the possibility that a directional direct search method (\ddsm) generates a sequence of trial points~$(x_k)_k$ converging to a point $x_*$ where~$f$ is discontinuous and whose objective function value~$f(x_*)$ is strictly less than $\lim_{k\to\infty} f(x_k)$.
Moreover the \ddsm generates no trial point in one of the two branches of~$f$ near $x_*$.
This note also investigates the proof of the theorem to highlight the inexact statements in the original paper.
Finally this work concludes with a modification of the \ddsm that allows to recover the properties broken by the counterexample.


\end{GDabstract}

%\begin{GDabstract}{R\'esum\'e}
%Cette note fournit un contre-exemple à un théorème proposé dans la dernière partie de l'article \textit{Analysis of direct searches for discontinuous functions}, Mathematical Programming Vol. 133, pp.~299--325, 2012.
%Le contre-exmple repose sur une fonction-objectif $f: \R \to \R$ qui satisfait toutes les hypothèses requires par le théorème mais contredit certaines de ses conclusions.
%Un corollaire au théorème est également affecté par le contre-exemple.
%Le principal problème révélé par le contre-exemple est la possibilité qu'une méthode de recherche directe (\ddsm) génère une suite d'optimiseurs $(x_k)_k$ convergeant vers un point $x_*$ en lequel $f$ est discontinue et telle que la valeur-objectif $f(x_*)$ est strictement inférieure à $\lim_{k\to\infty}f(x_k)$.
% De plus, la \ddsm ne génère aucun point dans l'une des deux branches de $f$ au voisinage de $x_*$.
%Cette note étudie également la preuve du théorème pour révéler des affirmations incorrectes dans l'article initial.
%Enfin, ce travail propose une modifiation de la \ddsm permettant d'obtenir les propriétés cassées par le contre-exemple.
%
%\paragraph{Mots cl\'es\,: } Optimisation discontinue, méthodes de recherche directe, dérivées généralisées.
%
%\end{GDabstract} 

\begin{GDacknowledgements}
Work of the first author is supported by NSERC Canada Discovery Grant 2020--04448.

\end{GDacknowledgements}

%% %%%%%%%%%%%%%%%%%%%%%%%%%%%%%%%%%%%%%%%%%%%%%%%%%
%% %%%%%%%%%%%%%%%% Article %%%%%%%%%%%%%%%%%%%%%%%%
%% %%%%%%%%%%%%%%%%%%%%%%%%%%%%%%%%%%%%%%%%%%%%%%%%%

\GDarticlestart
\newpage
%% %%%%%%%%%%%%%%%%%%%%%%%%%%%%%%%%%%%%%%%%%%%%%%%%%
\section{Introduction}
\label{sec:intro}
%% %%%%%%%%%%%%%%%%%%%%%%%%%%%%%%%%%%%%%%%%%%%%%%%%%

Derivative-free optimization is the mathematical field that focuses on nonsmooth optimization problems with lack of derivatives (see, \textit{e.g.}, \cite{AlAuGhKoLed2020,AuHa2017,CoScVibook}).
The generic problem is
\begin{equation*}
	\underset{x \in \R^n}{\minimize}~ f(x) \quad s.t. \quad x \in \Omega,
\end{equation*}
where $f:\R^n \to \R\cup\{+\infty\}$ is nonsmooth and the feasible region $\Omega$ is a nonempty subset of $\R^n$.
A class of derivative-free algorithms well studied in the literature is the so-called \textit{directional Direct Search Methods} (\ddsm).
When $f$ has bounded level sets, a \ddsm generates a so-called \textit{refined point}~$x_*$ and a set~$D$ of so-called \textit{refining unitary directions}.
In addition, when $f$ is locally Lipschitz-continuous at~$x_*$, it is known that the Clarke generalized derivative $f^\circ_C$ is nonnegative at~$x_*$ in each refining direction~$d \in D \cap H_\Omega(x_*)$, where $H_\Omega(x_*)$ stands for the hypertangent cone to $\Omega$ at~$x_*$.
Relaxing this last continuity assumption, and thus minimizing a possibly discontinuous objective function, is challenging.
The paper~\cite{ViCu2012} settles several results about the behavior of a \ddsm in that context.

The main results of this paper (precisely \cite[Theorem~3.1]{ViCu2012} and~\cite[Theorem~3.2]{ViCu2012}) show that, if a \ddsm generates a refining subsequence satisfying~$f(x_k) \to f(x_*)$, then the Rockafellar generalized derivative~$f^\circ_R$ (applicable to discontinuous functions, as opposed to the Clarke generalized derivative) is nonnegative at $x_*$ in all directions $d \in D \cap T_\Omega(x_*)$, where $T_\Omega(x_*)$ stands for the tangent cone to $\Omega$ at $x_*$.
These important results extend the previous analyses to discontinuous functions.

At the end of the paper, two additional results (precisely \cite[Theorem~4.1]{ViCu2012} and~\cite[Corollary~4.1]{ViCu2012}) attempt to derive sufficient conditions on the refining subsequence to guarantee that~$f(x_k) \to f(x_*)$.
Unfortunately, to the best of our knowledge, one of the statements from~\cite[Theorem~4.1]{ViCu2012} is incorrect, which also affects the validity of~\cite[Corollary~4.1]{ViCu2012}.
The present work focuses on this inexact statement.

A \ddsm iteratively repeats a two-steps process formalized in~\cite[Algorithm~2.1]{ViCu2012}.
First, an optional \search step evaluates some points anywhere in $\R^n$.
Second, a mandatory \poll step samples $f$ around the current incumbent solution $x_k$ according to some chosen step parameter $\alpha_k >0$ and set~$D_k$ of directions.
For a given sequence $(x_k;\alpha_k;D_k)_{k\in\N}$, a so-called \textit{refining subsequence} (indexed by~$K \subseteq \N$) to a \textit{refined point} $x_*$ and a set $D$ of \textit{refining unitary directions} mean that the three following conditions are satisfied: (i) each iteration indexed by~$k \in K$ is unsuccessful (that is, all \search and \poll trial points fail to dominate the incumbent solution); (ii) $(x_k)_{k \in K}$ converges to~$x_*$ and $(\alpha_k)_{k \in K}$ converges to~$0$; (iii) any direction in $D$ is the limit of a subsequence of refining normalized directions~$(\frac{d_k}{\norm{d_k}})_{k \in K_1}$ where~$0 \neq d_k \in D_k$ for all $k \in K_1 \subseteq K$.
Following this terminology,~\cite[Theorem~4.1]{ViCu2012} and~\cite[Corollary~4.1]{ViCu2012}, plus~\cite[Assumption~2.1]{ViCu2012} and~\cite[Assumption~4.1]{ViCu2012} they rely on, are stated as follows.

\begin{quote}\textbf{\cite[Assumption~2.1]{ViCu2012}}
	The level set $L(x_0) = \{x \in \Omega : f(x) \leq f(x_0)\}$ is bounded.
	The function $f$ is bounded below on $L(x_0)$.
\end{quote}

\begin{quote}\textbf{\cite[Assumption~4.1]{ViCu2012}}
	The function $f$ is such that there exists a neighborhood $B$ of $x_*$ (a limit point of a refining subsequence) which admits a finite partition $B = \cup_{i=1}^{n_B} B_i$ such that, for all $i \in \{1, \dots, n_B\}$,
	\begin{enumerate}
		\item $\int(B_i) \neq \emptyset$,
		\item $\cl(B_i)$ has the exterior cone property (see~\cite[Definition~4.1]{ViCu2012}),
		\item $f$ is Lipschitz-continuous in $\int(B_i)$ and can be continuously extended from $\int(B_i)$ to~$\partial B_i$.
	\end{enumerate}
\end{quote}

\begin{quote}\textbf{\cite[Theorem~4.1]{ViCu2012}}
	Consider a refining subsequence $(x_k)_{k \in K}$ converging to $x_* \in \Omega$ (and note that~\cite[Assumption~2.1]{ViCu2012} is required for the existence of such a subsequence).
	Assume that $f$ is lower semicontinuous at $x_*$ and satisfies~\cite[Assumption~4.1]{ViCu2012}.
	Let the sets of refining directions for $x_*$ corresponding to any infinite subsequence of $K$ be dense in the unit sphere.
	
	If $x_*$ belongs to the interior of a partition set in $\{B_1, \dots, B_{n_B}\}$, then $f^\circ_C(x_*;v) \geq 0$ for all refining directions $v \in T_\Omega(x_*)$ (assuming here also that $H_\Omega(x_*)$ is nonempty).
	
	Otherwise, there exists a subsequence $K' \subset K$ and a partition set $B' \in \{B_1, \dots, B_{n_B}\}$ such that (i) $(x_k)_{k \in K} \subset \cl(B')$, (ii) there is an infinite number of poll points, corresponding to iterates in $K$, in $\int(B')$, and (iii) there is an infinite number of poll points, corresponding to iterates in $K$, in $\R^n \setminus \cl(B')$.
\end{quote}

\begin{quote}\textbf{\cite[Corollary~4.1]{ViCu2012}}
	Under~\cite[Assumption~2.1]{ViCu2012} and the assumptions of~\cite[Theorem~4.1]{ViCu2012} and when $n_B = 2$, there exists a subsequence $K_* \subset K$ and a partition set $B_* \in \{B_1, B_2\}$ such that, when $x_*$ is in the border of the two partition sets,
	\begin{enumerate}
		\item $B_*$ satisfies the properties stated for $B'$ in~\cite[Theorem~4.1]{ViCu2012},
		\item $\lim_{k \in K_*} f(x_k) = f(x_*)$.
	\end{enumerate}
\end{quote}

The present note provides in Section~\ref{sec:counterexample} a counterexample to~\cite[Theorem~4.1]{ViCu2012} and~\cite[Corollary~4.1]{ViCu2012}. 
More precisely, Section~\ref{sec:counterexample/f} constructs an unconstrained optimization problem with a single-variable real-valued objective function $f : \R \to \R$, Section~\ref{sec:counterexample/mads} proposes a specific \ddsm instance, and Section~\ref{sec:counterexample/theorems_contradicted} shows that the counterexample satisfies all the assumptions required in~\cite[Theorem~4.1]{ViCu2012} and~\cite[Corollary~4.1]{ViCu2012} but contradicts some of their conclusions.
An analysis of the proofs of~\cite[Theorem~4.1]{ViCu2012} and~\cite[Corollary~4.1]{ViCu2012} is provided in Section~\ref{sec:problems_proofs} to identify the incorrect steps.

A second objective of the present work is to propose in Section~\ref{sec:new_framework} an adaptation of~\cite[Algorithm~2.1]{ViCu2012} which recovers the properties broken by the counterexample.
Section~\ref{sec:new_framework/mads} introduces this adaptation, relying on a strategy proposed in~\cite{AuBaKo22}.
It is based on the evaluation of a few additional points at each iteration to asymptotically evaluate a dense set of points around $x_*$, thus it evaluates points in each branch of~$f$ near~$x_*$ and reveals any discontinuity.
This leads to Theorem~\ref{th:main_result}, stated and proved in Section~\ref{sec:new_framework/analysis}, which claims results similar to~\cite[Corollary~4.1]{ViCu2012} under an assumption similar to~\cite[Assumption~4.1]{ViCu2012} but without requiring~$n_B = 2$.

%% %%%%%%%%%%%%%%%%%%%%%%%%%%%%%%%%%%%%%%%%%%%%%%%%%
\section{A single-variable unconstrained counterexample}
\label{sec:counterexample}
%% %%%%%%%%%%%%%%%%%%%%%%%%%%%%%%%%%%%%%%%%%%%%%%%%%

The single-variable real-valued objective function~$f : \R \to \R$ in this counterexample is illustrated in the left part of Figure~\ref{fig:f}.
It possesses local minima at $x = \frac{5}{4}\times 2^\ell$ for each integer $\ell \in \Z$.
The \ddsm parameters are chosen so that each incumbent solution lies at one of these local minimizers and the corresponding step parameters are sufficiently small so that each \poll step results in a failure.
A \search step is added only at odd-index iterations, generating a trial point at the next local minimizer to the left of the current incumbent solution.
The poll size parameter is halved at unsuccessful iterations, and remains constant at successful ones.
Hence, by construction, all even-index iterations fail, while all odd-index iterations succeed, and the sequence of the incumbent solutions converges to the origin.

\begin{figure}[ht]
	\centering
	\includegraphics[width=0.9\linewidth]{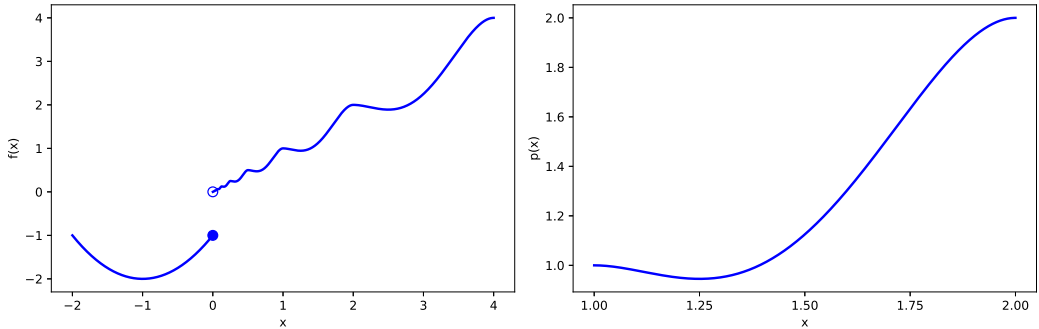}
	\caption{The function $f$ (left plot), and the polynomial $p$ (right plot) used to construct $f$.}
	\label{fig:f}
\end{figure}

%% %%%%%%%%%%%%%%%%%%%%%%%%%%%%%%%%%%%%%%%%%%%%%%%%%
\subsection{Construction of the objective function}
\label{sec:counterexample/f}
%% %%%%%%%%%%%%%%%%%%%%%%%%%%%%%%%%%%%%%%%%%%%%%%%%%

This section describes the construction of the discontinuous objective function $f : \R \to \R$ illustrated in the left plot of Figure~\ref{fig:f}.
On the closed interval $\R^-$, it is defined as the convex quadratic function given by~$f(x) = (x+1)^2-2$ for all~$x \in \R^-$.
On the open interval $\R^+_*$, it is constructed from the polynomial~$p : [1,2] \to \R$ illustrated in the right plot of Figure~\ref{fig:f} and given by
\begin{equation*}
	\label{eq:p}
	p(x) = -18 + 60x -69 x^2 + 34 x^3 - 6x^4 ~\mbox{ for all}~ x \in [1,2].
\end{equation*}
It satisfies $p(1) = 1$, $p(2) = 2$, $p'(1) = p'\left(\frac{5}{4}\right) = p'(2) = 0$, $p''(1) < 0,~ p''(2) < 0$, $p''\left(\frac{5}{4}\right) > 0$. Hence~$p$ has a unique local minimizer at $x = \frac{5}{4}$ and two local maximizers at the endpoints of the interval~$[1,2]$.
The objective function $f$ is constructed on $\R^+_*$ by scaling the polynomial $p$. Precisely, $f$ equals to $2p$ on the interval $2\times[1,2[$, to $p$ on $[1,2[$, to $\frac{p}{2}$ on $\frac{1}{2}\times[1,2[$ and, in general, to $2^\ell p$ on the interval~$2^\ell\times[1,2[$ with~$\ell \in \Z$.
Hence, on each interval $2^\ell \times [1, 2]$ with $\ell \in \Z$, $f$ has exactly three stationary points (given by a unique local minimizer at~$x = \frac{5}{4} \times 2^\ell$ and two local maximizers at the endpoints).

To summarize, consider the unconstrained optimization problem ($\Omega = \R$) given by
\begin{equation*}
	\label{eq:problem}
	\underset{x ~\in~ \R}{\minimize}~ f(x),
\end{equation*}
where
\begin{equation*}
	f(x) \ = \
	\left\{\begin{array}{ll}
		(x+1)^2-2
		& \mbox{if}~ x \leq 0, \\[5pt]
		2^\ell p\left(\dfrac{x}{2^\ell}\right)
		& \mbox{if}~ x \in 2^\ell\times[1,2[ \mbox{ with } \ell \in \Z, \\
	\end{array}\right. ~\mbox{ for all}~ x \in \R.
\end{equation*}

\vspace{0.1cm}

\begin{proposition}
	\label{prop:behavior_f}
	The function $f$ satisfies~\cite[Assumption~2.1]{ViCu2012} and~\cite[Assumption~4.1]{ViCu2012} with $x_* = 0$.
\end{proposition}

\begin{proof}
	\cite[Assumption~2.1]{ViCu2012} is satisfied and~\cite[Assumption~4.1]{ViCu2012} holds for $x_* = 0$.
	Indeed, $x_* = 0$ has a neighborhood $B = {]}-2,2[$ that can be partitioned as $B = B_1 \cup B_2$ with $B_1 = {]}-2,0]$ and $B_2 = {]}0,2[$.
	Then, for both $i \in \{1,2\}$, $\int(B_i) \neq \emptyset$, $f$ is Lipschitz-continuous on $B_i$ and can be continuously extended to $\partial B_i$, and $\cl(B_i)$ has the exterior cone property.
\end{proof}

%% %%%%%%%%%%%%%%%%%%%%%%%%%%%%%%%%%%%%%%%%%%%%%%%%%
\subsection{Construction of the \ddsm instance}
\label{sec:counterexample/mads}
%% %%%%%%%%%%%%%%%%%%%%%%%%%%%%%%%%%%%%%%%%%%%%%%%%%

This section describes a \ddsm instance to minimize $f$.
The starting point is $x_0 = \frac{5}{4}$ and the initial step parameter is $\alpha_0 = \frac{1}{4}$.
The \poll sets of directions are $D_k = \{-1,1\}$ for all $k \in \N$.
The step parameter is updated as $\alpha_{k+1} = \alpha_k$ if the iteration $k$ finds a better point than the incumbent solution~$x_k$, and~$\alpha_{k+1} = \frac{\alpha_k}{2}$ otherwise.
The \search step generates no trial point when the iteration number $k$ is even, and returns a single trial point $t_k = x_k - 5\alpha_k$ when $k$ is odd.

These parameters are chosen so that, at each even iteration $k=2q$ with $q \in \N$, the incumbent solution~$x_{2q}$ is the local minimum of $f$ on the interval $2^{-q} \times [1,2]$, the \poll step fails because all its trial points lie in that interval, and the \search step running on the following odd iteration $2q+1$ generates the trial point corresponding to the next local minimizer on the left of $x_{2q}$ and located on the interval~$2^{-(q+1)}\times[1,2]$.
Hence the sequence of incumbent solutions $(x_k)_{k\in\N}$ converges to the origin, the sequence of step parameters~$(\alpha_k)_{k\in\N}$ converges to $0$, and both directions $d_1 = -1$ and $d_2 = 1$ are considered at each even iteration.
These claims are formalized in the next result.

\vspace{0.1cm}

\begin{lemma}
	\label{prop:expression_sequences_mads}
	The sequences of incumbent solutions $(x_k)_k$ and of step parameters $(\alpha_k)_k$ satisfy
	\begin{equation*}
		\forall q \in \N,
		\quad
		x_{2q} = \frac{5}{4}\times 2^{-q} = x_{2q+1},
		\quad
		\alpha_{2q} = \frac{1}{4}\times 2^{-q} = 2\alpha_{2q+1}.
	\end{equation*}
	At each even iteration $k=2q$ with $q \in \N$, the \poll step fails after evaluating both trial points~$x_{2q}+\alpha_{2q}d_1$ and~$x_{2q}+\alpha_{2q}d_2$.
\end{lemma}

\begin{proof}
	The proof is done by induction over $q \in \N$.
	
	Iteration $0$ starts with $x_0 = \frac{5}{4}$ and $\alpha_0 = \frac{1}{4}$.
	It is an even iteration and, therefore, there is no \search step.
	The incumbent solution $x_0$ is a local minimizer on the interval $[x_0 - \alpha_0, x_0 + \alpha_0] = [1, \frac{3}{2}]$ and, therefore, the two candidates $x_0+\alpha_0d_1$ and $x_0+\alpha_0d_2$ generated by the \poll step fail to dominate $x_0$.
	Thus iteration $0$ is unsuccessful, hence $x_1 = x_0 = \frac{5}{4}$ and $\alpha_1 = \frac{1}{2}\alpha_0 = \frac{1}{4}\times 2^{-1}$.
	Then the induction statement holds for $q = 0$.
	
	Now consider $q \in \N_*$ and assume that the induction statement holds for $q-1$.
	This implies that iteration $2q-1$ starts with $x_{2q-1} = x_{2(q-1)+1} = \frac{5}{4}\times 2^{-(q-1)}$ and $\alpha_{2q-1} = \alpha_{2(q-1)+1} = \frac{1}{4}\times 2^{-q}$.
	It is an odd iteration, so the \search step generates the trial point $t_{2q-1} = x_{2q-1} - 5\alpha_{2q-1} = \frac{5}{4}\times 2^{-(q-1)} -\frac{5}{4}\times 2^{-q} = \frac{5}{4}\times 2^{-q}$.
	This trial point satisfies $f(t_{2q-1}) = \frac{f(x_{2q-1})}{2} < f(x_{2q-1})$. Hence iteration $2q-1$ is a success and iteration $2q$ starts with~$x_{2q} = t_{2q-1} = \frac{5}{4}\times 2^{-q}$ and $\alpha_{2q} = \alpha_{2q-1} = \frac{1}{4}\times 2^{-q}$.
	This is an even iteration and, therefore, there is no \search step.
	The incumbent solution $x_{2q}$ is a local minimizer on the interval~$[x_{2q} - \alpha_{2q}, x_{2q} + \alpha_{2q}] = 2^{-q}\times[1, \frac{3}{2}]$ and then, the two candidates $x_{2q}+\alpha_{2q}d_1$ and $x_{2q}+\alpha_{2q}d_2$ generated by the \poll step fail to dominate $x_{2q}$.
	Thus iteration $2q$ is unsuccessful, hence $x_{2q+1} = x_{2q} = \frac{5}{4}\times 2^{-q}$ and $\alpha_{2q+1} = \frac{1}{2}\alpha_{2q} = \frac{1}{4}\times 2^{-(q+1)}$.
	Then the induction statement holds for $q \in \N_*$ if it holds for $q-1$, which completes the proof.
\end{proof}

The sequence $(x_k;\alpha_k;D_k)_{k\in\N}$ constructed above satisfies the requirements of~\cite[Theorem~4.1]{ViCu2012} and~\cite[Corollary~4.1]{ViCu2012}. This claim is formalized in the next proposition.

\vspace{0.1cm}

\begin{proposition}
	\label{prop:behavior_mads}
	The sequence $(x_k;\alpha_k;D_k)_{k \in \N}$ is a refining sequence which converges to the refined point~$x_* = 0$ and the refining set of directions $D = \{-1,1\}$ is dense in the unit sphere.
\end{proposition}

\begin{proof}
	Lemma~\ref{prop:expression_sequences_mads} proves that the sequence of incumbent solutions $(x_k)_{k \in \N}$ converges to~$x_* = 0$ and that the sequence of poll parameters $(\alpha_k)_{k \in \N}$ converges to~$0$. 
	Furthermore, for each even value of~$k$, the \poll step evaluates $x_k+\alpha_kd$ for all~$d \in D_k = \{-1,1\}$.
    Thus the set of refining directions is~$D = \{-1,1\}$.
\end{proof}

%% %%%%%%%%%%%%%%%%%%%%%%%%%%%%%%%%%%%%%%%%%%%%%%%%%
\subsection{Contradiction to some conclusions of~\cite[Theorem~4.1]{ViCu2012} and~\cite[Corollary~4.1]{ViCu2012}}
\label{sec:counterexample/theorems_contradicted}
%% %%%%%%%%%%%%%%%%%%%%%%%%%%%%%%%%%%%%%%%%%%%%%%%%%

Propositions~\ref{prop:behavior_f} and~\ref{prop:behavior_mads} show that the counterexample satisfies all the requirements of~\cite[Theorem~4.1]{ViCu2012} and~\cite[Corollary~4.1]{ViCu2012}.
In that context~\cite[Theorem~4.1]{ViCu2012} claims that there exist a subsequence (indexed by $K' \subseteq \N$) and $i \in \{1,2\}$ such that the \poll step fails at each iteration $k \in K'$ and
\begin{equation*}
	(a):\quad
	\forall k \in K',~
	\left\{\begin{array}{l}
		x_k             \in \cl(B_i), \\
		x_k+\alpha_kd_2 \in \int(B_i),
	\end{array}\right.
	\qquad \mbox{and} \qquad
	(b):\quad
	\forall k \in K',~
	\left\{\begin{array}{l}
		x_k             \in \cl(B_i), \\
		x_k+\alpha_kd_1 \in \R \setminus \cl(B_i).
	\end{array}\right.
\end{equation*}

In addition~\cite[Corollary~4.1]{ViCu2012} asserts that
\begin{equation*}
	(c):\quad \underset{k \in K}{\lim}~ f(x_k) = f(x_*).
\end{equation*}

In the counterexample, the sequence of poll trial points has no element in~$B_1 = {]}-2,0]$ containing the refined point $x_* = 0$ on its border.
Thus, the claim $(a)$ holds while the claims $(b)$ and $(c)$ are false.
Indeed, by restricting to $i = 2$ (since $x_k \notin cl(B_1)$ for all $k \in \N$) and to the indexes~$K = 2\N$ where the \poll step fails, the claim $(a)$ holds with any $K' \subseteq K$.
However, for any $K' \subseteq K$, $\lim_{k \in K'} f(x_k) = 0 \neq -1 = f(x_*)$ and the trial point $x_k+\alpha_kd_1 = 2^{-q}$ belongs to $B_2$ for all $k=2q \in K'$ with~$q \in \N$.
Hence the claims $(b)$ and $(c)$ are incorrect.

%% %%%%%%%%%%%%%%%%%%%%%%%%%%%%%%%%%%%%%%%%%%%%%%%%%
\section{Analysis of the proof leading to incorrect conclusions}
\label{sec:problems_proofs}
%% %%%%%%%%%%%%%%%%%%%%%%%%%%%%%%%%%%%%%%%%%%%%%%%%%

The current section investigates the proof of~\cite[Theorem~4.1]{ViCu2012} in its general context.
In the paper~\cite{ViCu2012}, this proof is divided into six paragraphs that are summarized in the following six steps. An additional seventh step leads to~\cite[Corollary~4.1]{ViCu2012}.
The error in the proof lies in the fifth step.

Denote $i \in \ll 1,n_B \rr$ such that $x_*$ belongs to the partition set $B_i$.
First, the authors restrict the proof to the \mads algorithm since the proof can be easily adapted to other \ddsm.
Second, they address the case where $x_* \in \int(B_i)$.
Third, when $x_* \in \partial B_i$, the authors address first the subcase where all incumbent solutions~$x_k$ lie in the borders of the partition sets.
In the remaining subcases, they extract a subsequence (indexed by $K_1 \subseteq K$) such that~$x_k \in \int(B_i)$ for all $k \in K_1$.
Since the set of refining directions is dense in the unit sphere, they extract a second subsequence (indexed by $K_2 \subseteq K_1$) such that, at any iteration $k \in K_2$, some \poll trial points belong to $\int(B_i)$.
Fourth, the authors create a function $\overline{f}$ Lipschitz-continuously extending $f$ from $B_i$ to $\cl(B_i)$.
To the best of our knowledge, these four steps are all valid.

Then, in the fifth paragraph, the authors aim to prove by contradiction that some of the \poll points associated with iterations $k \in K_2$ do not belong to $\cl(B_i)$.
Their proof involves a function~$\widetilde{f}$ extending $\overline{f}$ from $B_i$ to $\R^n$ such that $\widetilde{f} = \overline{f}$ on $B_i$, $\widetilde{f}$ is Lipschitz-continuous on $\R^n$ and~$\widetilde{f}$ is locally strictly decreasing at any $z \in \partial B_i$ in all directions belonging to a cone with nonempty interior.
To the best of our knowledge, the construction of $\widetilde{f}$ is valid.
However, the authors consider a direction $d$ on which $\widetilde{f}$ is strictly decreasing at $x_*$ and erroneously claim that the Clarke derivative $\widetilde{f}^\circ_C(x_*;d)$ is thus strictly negative.
The error in the proof lies in this incorrect claim.
Indeed, the fact that $\widetilde{f}$ is strictly decreasing at $x_*$ in a direction $d$ does not imply that $\widetilde{f}^\circ_C(x_*;d)$ is strictly negative (for example, the single-variable function $x \mapsto -\abs{x}$ has a Clarke derivative equaling $1$ at $x = 0$ in both directions $d = 1$ and $d = -1$).
The following quote highlights the error.

\begin{quote}\textbf{[Erroneous claim in the fifth paragraph of the proof of~\cite[Theorem~4.1]{ViCu2012}]}
	Let us assume that all poll points associated with the refining subsequence belong to $\cl(B_i)$.
	We will see that this leads us to a contradiction.
	[\dots]
	Let $\widetilde{f}$ be the extended function (and $L$ its Lipschitz constant).
	We then obtain that
	\begin{equation*}
		\widetilde{f}^\circ_C(x_*;v) \geq [\dots] \geq 0,
	\end{equation*}
    for all refining directions $v$, which is a contradiction since these directions are dense in the unit sphere and $\widetilde{f}$ is locally strictly decreasing from $x_*$ along all directions in a cone of nonempty interior.
\end{quote}

The sixth paragraph, which considers the above claim as true, deduces that there exists a third subsequence (indexed by $K_3 \subseteq K_2$) for which the \poll step at iterations $k \in K_3$ generates trial points~$t_k = x_k+\alpha_kd_k$ satisfying $f(t_k) \geq f(x_k)$ and $t_k \notin \cl(B_i)$.
This concludes the proof of~\cite[Theorem~4.1]{ViCu2012}.
Finally a seventh paragraph, also relying directly on the above claim, leads to~\cite[Corollary~4.1]{ViCu2012}.
In the case $n_B = 2$, and arbitrarily denoting~$i = 2$, the authors claim that some poll trial points belong to $B_2$ and some to $B_1$ for all $k \in K_3$.
Then, the limit of the sequence $(f(x_k))_{k \in K_3}$ equals $f(x_*)$ via the lower semicontinuity of $f$ at $x_*$ and the failing \poll steps at iterations $k \in K_3$.
To the best of our knowledge, these last two steps are incorrect without assuming the inexact claim from the fifth paragraph. 
Therefore their conclusions are not valid in general.

%% %%%%%%%%%%%%%%%%%%%%%%%%%%%%%%%%%%%%%%%%%%%%%%%%%
\section{An adapted \ddsm to handle the counterexample}
\label{sec:new_framework}
%% %%%%%%%%%%%%%%%%%%%%%%%%%%%%%%%%%%%%%%%%%%%%%%%%%

The counterexample introduced in Section~\ref{sec:counterexample} shows that a sequence of incumbent solutions generated by a \ddsm may approach a discontinuity without noticing its existence.
Section~\ref{sec:new_framework/mads} proposes an adaptation of~\cite[Algorithm~2.1]{ViCu2012} (see Algorithm~\ref{algo:new_mads}) which avoids this drawback, as shown by the new convergence Theorem~\ref{th:main_result} stated and proved in Section~\ref{sec:new_framework/analysis}.

%% %%%%%%%%%%%%%%%%%%%%%%%%%%%%%%%%%%%%%%%%%%%%%%%%%
\subsection{A revised algorithm with a \revpoll step}
\label{sec:new_framework/mads}
%% %%%%%%%%%%%%%%%%%%%%%%%%%%%%%%%%%%%%%%%%%%%%%%%%%

This section adds a second \poll step to~\cite[Algorithm~2.1]{ViCu2012} to overcome the aforementioned drawback.
This so-called \revpoll step has been introduced in the literature as a part of the \textsc{Disco}\mads algorithm~\cite{AuBaKo22} (which is a variant of a \ddsm called the \mads algorithm) designed for a framework differing from the present work (including additional mechanisms to push incumbent solutions away from discontinuities).
At each iteration, the \revpoll step evaluates at least one random mesh point in a ball of constant radius centered at the current incumbent solution.
Therefore, when the sequence of incumbent solutions approaches a refined point $x_*$ at which $f$ is discontinuous and when the mesh becomes sufficiently fine, the \revpoll step eventually evaluates points in each branch of $f$ that has a nonempty interior near $x_*$.
This property allows to correct and to strengthen the convergence analysis provided in~\cite{ViCu2012}, as shown in the next Section~\ref{sec:new_framework/analysis}.

Algorithm~\ref{algo:new_mads} below encloses the addition of a \revpoll step to~\cite[Algorithm~2.1]{ViCu2012}.
Precisely, the \textbf{\search step}, the \textbf{\poll step} and the \textbf{Step size parameter update} follow~\cite[Algorithm~2.1]{ViCu2012}.
Also, the forcing function $\overline{\rho}$ (assumed to be continuous and nondecreasing on $\R^+$ with~$\rho(t)/t \to 0$ as~$t \searrow 0$) is extracted from~\cite[Algorithm~2.1]{ViCu2012}. Hence the only novelty lies in the \textbf{\revpoll step}, with its parameter $R>0$ in the \textbf{Initialization} step.

\begin{algorithm}[h]
	\caption{Directional direct-search method (following~\cite[Algorithm~2.1]{ViCu2012}, except for the additional \textbf{\revpoll step})}
	\label{algo:new_mads}
	\begin{algorithmic}
		\State \textbf{Initialization} \vspace{0.1cm}
		\State \hfill \begin{minipage}{0.96\linewidth}
			Choose $x_0 \in \Omega$ with $f(x_0) < +\infty$, $\alpha_0 > 0$, $0 < \beta_1 \leq \beta_2 < 1$, $\gamma \geq 1$ and a constant radius~$R > 0$
			for the \revpoll step.
			Let $\mathbb{D}$ be a (possibly infinite) set of positive spanning sets.
		\end{minipage} \vspace{0.2cm}
		\For{k = 0, 1, 2, \dots}
		\State \textbf{\search step} \vspace{0.1cm}
		\State \hfill \begin{minipage}{0.92\linewidth}
			Define a finite (possibly empty) set $S_k \subset \R^n$.
			Denote~$t_k \in \argmin \{f(x_k)-\overline{\rho}(\norm{t-x_k}) : t \in S_k \cap \Omega\}$.
			If $f(t_k) < f(x_k)-\overline{\rho}(\norm{t-x_k})$, set~$x_{k+1} = t_k$ and skip the next two steps.
		\end{minipage} \vspace{0.2cm}
		\State \textbf{\revpoll step} \vspace{0.1cm}
		\State \hfill \begin{minipage}{0.92\linewidth}
			Randomly define a nonempty finite set $D_k^r$ according to the uniform independent distribution over the closed ball of radius $R$ centered at $0$.
			Denote~$P^r_k = \{x_k+d : d \in D^r_k\}$ and~$t_k \in \argmin \{f(x_k)-\overline{\rho}(\norm{t-x_k}) : t \in P^r_k \cap \Omega\}$.
			If $f(t_k) < f(x_k)-\overline{\rho}(\norm{t-x_k})$, set~$x_{k+1} = t_k$ and skip the next \poll step.
		\end{minipage} \vspace{0.2cm}
		\State \textbf{\poll step} \vspace{0.1cm}
		\State \hfill \begin{minipage}{0.92\linewidth}
			Choose $D_k \in \mathbb{D}$.
			Denote~$P_k = \{x_k+\alpha_kd : d \in D_k\}$ and~$t_k \in \argmin \{f(x_k)-\overline{\rho}(\norm{t-x_k}) : t \in P_k \cap \Omega\}$.
			If $f(t_k) < f(x_k)-\overline{\rho}(\norm{t-x_k})$, then set $x_{k+1} = t_k$.
			Otherwise, set $x_{k+1} = x_k$.
		\end{minipage} \vspace{0.2cm}
		\State \textbf{Step size parameter update} \vspace{0.1cm}
		\State \hfill \begin{minipage}{0.92\linewidth}
			If $x_{k+1} \neq x_k$, increase the step size parameter as $\alpha_{k+1} \in [\alpha_k,\gamma\alpha_k]$.
			Otherwise strictly decrease the step size parameter as $\alpha_{k+1} \in [\beta_1\alpha_k,\beta_2\alpha_k]$.
		\end{minipage}
		\EndFor
	\end{algorithmic}
\end{algorithm}

Algorithm~\ref{algo:new_mads} evaluates the trial points $\mathcal{P}_k = P^r_k \cup P_k$ and trial directions $\D_k = D^r_k \cup D_k$ at each unsuccessful polling iteration $k \in \N$.
Hence, the definition of a refining sequence is adapted as follows.
A \textit{refining subsequence} $(x_k;\alpha_k;\D_k)_{k \in K}$ with corresponding \textit{refined point} $x_*$ and set $\D$ of \textit{refining unitary directions} satisfy the three following conditions: (i) $f(t) \geq f(x_k) - \overline{\rho}(\norm{t-x_k})$ for all~$t \in S_k \cup \mathcal{P}_k$; (ii) $(x_k)_{k \in K}$ converges to~$x_*$ and $(\alpha_k)_{k \in K}$ converges to $0$; (iii) any direction in $D$ is the limit of a subsequence of refining normalized directions~$(\frac{d_k}{\norm{d_k}})_{k \in K_1}$ where~$0 \neq d_k \in \D_k$ for all $k \in K_1 \subseteq K$.

%% %%%%%%%%%%%%%%%%%%%%%%%%%%%%%%%%%%%%%%%%%%%%%%%%%
\subsection{Analysis of the revised algorithm}
\label{sec:new_framework/analysis}
%% %%%%%%%%%%%%%%%%%%%%%%%%%%%%%%%%%%%%%%%%%%%%%%%%%

Algorithm~\ref{algo:new_mads} inherits the properties of~\cite[Algorithm~2.1]{ViCu2012} since the \revpoll step can be viewed as a specific \search step.
This section proves that Algorithm~\ref{algo:new_mads} has an additional asymptotic result sufficiently strong to correct and generalize~\cite[Theorem~4.1]{ViCu2012} and~\cite[Corollary~4.1]{ViCu2012}.
This result is stated below as Theorem~\ref{th:main_result} and is based on the next Lemma~\ref{prop:behavior_revealing_poll}, which is similar to~\cite[Lemma~4.12]{AuBaKo22}, and on Assumption~\ref{hyp:closure_branches} which replaces~\cite[Assumption~4.1]{ViCu2012}.
In the sequel, $\B_r(x)$ denotes the closed ball of~$\R^n$ of radius $r > 0$ centered at $x \in \R^n$, and $\S^n$ denotes the unit sphere in $\R^n$.

\vspace{0.1cm}

\begin{lemma}
	\label{prop:behavior_revealing_poll}
	Let $(x_k;\alpha_k;\D_k)_{k \in K}$ be a refining subsequence generated by Algorithm~\ref{algo:new_mads} with refined point~$x_*$ and set $\D$ of refining unitary directions.
	The set $\cup_{k \in K} \P_k$ is almost surely dense in the closed ball~$\B_R(x_*)$ and $\D$ is almost surely dense in $\S^n$.
\end{lemma}

\begin{proof}
	This result is proved in~\cite[Lemma~4.12]{AuBaKo22} in the case of the \mads algorithm.
	The current proof focuses on the general case of a \ddsm.
	Let $\varepsilon > 0$.
	
	First, let $u \in \S^n$ and let us prove that $\B_\varepsilon(u)$ contains a direction $d \in \D$.
	Recall that each element of the infinite set~$\cup_{k \in K} D^r_k$ is drawn from the uniform independent distribution on $\B_R(0)$.
	Then, almost surely there exists an infinite subset of $K$ denoted $L$ such that for all $k \in L$, the random event~$\{\frac{d}{\norm{d}} : d \in D^r_k\} \cap \B_\varepsilon(u) \neq \emptyset$ happens.
	Let $(\frac{d_k}{\norm{d_k}})_{k \in L}$ with $d_k \in D^r_k$ and $\frac{d_k}{\norm{d_k}} \in \B_\varepsilon(u)$ for all~$k \in L$.
	Since $\B_\varepsilon(u)$ is a compact set, $(\frac{d_k}{\norm{d_k}})_{k \in L}$ has an accumulation unitary direction $d$ belonging to~$\D \cap \B_\varepsilon(u)$.
	
	Second, let $y \in \B_R(x_*)$, and let us prove that there is a \revpoll trial point $x \in \cup_{k \in K}P^r_k$ that belongs to $\B_\varepsilon(y)$.
	Since $x_*$ is a refined point of $(x_k)_{k \in K}$, there exists a subsequence (indexed by~$K_1 \subseteq K$) such that~$\norm{x_k-x_*} \searrow 0$ with $k \in K_1$.
	Thus, there exist another subsequence (indexed by~$K_2 \subseteq K_1$) such that the distance~$d(y, \B_R(x_k)) \searrow 0$ and an index $k_0 \in K_2$ such that $d(y, \B_R(x_k)) \leq \frac{\varepsilon}{2}$ for all $k \in K_3 = \{k : k \in K_2,~ k \geq k_0\}$.
	Then, denoting $\Theta_k = \int(\B_R(x_*) \cap \B_\varepsilon(y) \cap \B_R(x_k))$ for all~$k \in K_3$ and $\mathrm{vol}(\cdot)$ the volume of a set, one has
	\begin{equation*}
		\emptyset \ \neq \ \int(\Theta_k) \ = \  \Theta_k \ \subset \  \B_R(x_k)
		\quad \mbox{and} \quad
		0 \ \leq \ 
		\dfrac
		{\mathrm{vol}\left(\B_R(x_k)\setminus\Theta_k\right)}
		{\mathrm{vol}\left(\B_R(x_k)\right)}
		\ \leq \ 
		\dfrac
		{\mathrm{vol}\left(\B_R(x_{k_0})\setminus\Theta_{k_0}\right)}
		{\mathrm{vol}\left(\B_R(x_{k_0})\right)}
		\ < \ 1,
	\end{equation*}
	for all $k \in K_3$.
	Indeed, $\Theta_k$ is a nonempty open subset of $\B_R(x_k)$ for all $k \in K_3$.
    The second inequality follows from~$\mathrm{vol}(\B_R(x_k)) = \mathrm{vol}(\B_R(x_{k_0}))$ and $\B_R(x_k)\setminus\Theta_k = \B_R(x_k)\setminus(A\cup\B_\varepsilon(y))$ and $\norm{x_k-x_*} \searrow 0$ for all~$k \in K_3$.
	The first and third inequalities follow from $0 \leq \mathrm{vol}(\B_R(x_k)\setminus\Theta_k) < \mathrm{vol}(\B_R(x_k))$, since $\Theta_k$ is an open set included in $\B_R(x_k)$ for all $k \in K_3$.
	
	Denoting $E_k$ the random event $P_k^r \cap \Theta_k \neq \emptyset$ conditionally to $x_k$ for any $k \in K_3$, the proof is complete if at least one of the events $(E_k)_{k \in K_3}$ happens.
	The probability of any of the events $E_k$ is
    \begin{equation*}
		\mathbb{P}\left(E_k\right)
		\ = \
		1 - \left(
		\dfrac{\mathrm{vol}\left(\B_R(x_k)\setminus\Theta_k\right)}
		{\mathrm{vol}\left(\B_R(x_k)\right)}
		\right)^{\#P_k^r}
		\ \geq \
		1 -
		\dfrac{\mathrm{vol}\left(\B_R(x_{k_0})\setminus\Theta_{k_0}\right)}
		{\mathrm{vol}\left(\B_R(x_{k_0})\right)}
		\ = \ c \ > \ 0,
	\end{equation*}
    where $\#P_k^r \geq 1$ denotes the number of trial points in the \revpoll step at iteration $k$.
	Thus, for any $k \in K_3$, $\mathbb{P}(E_k)$ is bounded below by a strictly positive constant $c$.
	Since moreover the events~$(E_k)_{k \in K_3}$ are all independent, then almost surely at least one of them happens.
\end{proof}

\begin{assumption}
    \label{hyp:closure_branches}
    Any refined point has a neighborhood $B$ with a finite partition $B = \cup_{i=1}^{n_B} B_i$ satisfying the first and third statements of~\cite[Assumption~4.1]{ViCu2012} and $\cl(B_i) = \cl(\int(B_i))$ for each $i \in \ll1,n_B\rr$.
\end{assumption}

Assumption~\ref{hyp:closure_branches} replaces~\cite[Assumption~4.1]{ViCu2012}, changing the requirement involving the exterior cone property on each branch of $f$ by a requirement that no branch has a meagre subset non-adherent to its interior.
Then, with the \revpoll step and Lemma~\ref{prop:behavior_revealing_poll}, the following Theorem~\ref{th:main_result} holds.

\begin{theorem}
	\label{th:main_result}
	Under~\cite[Assumption~2.1]{ViCu2012} and Assumption~\ref{hyp:closure_branches}, let $x_* \in \Omega$ be the refined point and~$\D$ be the set of refining unitary directions of a refining subsequence $(x_k;\alpha_k;\D_k)_{k \in K}$ generated by Algorithm~\ref{algo:new_mads}.
	Then $\D$ is almost surely dense in $\S^n$.
	If $f$ is lower semicontinuous at $x_*$, then
	the partition set~$B_* \in \{B_1,\dots,B_{n_B}\}$ for which $x_* \in B_*$ satisfies
	\begin{enumerate}
		\item if $x_* \in \int(B_*)$, then $\lim_{k \in K}f(x_k) = f(x_*)$ and $f^\circ_C(x_*;d) \geq 0$ for all $d \in \D \cap T_\Omega(x_*)$;
		\item otherwise $x_* \in \partial B_*$ and there exists a subsequence (indexed by $K_* \subseteq K$) such that $x_k \in \cl(B_*)$ for all $k \in K_*$, $\lim_{k \in K_*} f(x_k) = f(x_*)$, and $\cup_{k \in K_*}\P_k$ is almost surely dense in $B \cap \B_R(x_*)$.
	\end{enumerate}
\end{theorem}

\begin{proof}
	The first three steps mentioned in Section~\ref{sec:problems_proofs} remain valid in the present framework.
    The remaining steps are not required.
	Thanks to these three steps, the proof is restricted to $\overline{\rho} \equiv 0$, the case where $x_* \in \int(B_*)$ is proved and, when $x_* \in \partial B_*$, there exists a subsequence (indexed by~$K_1 \subseteq K$) such that $x_k \in \cl(B_*)$ for all $k \in K_1$.
	Recall that, in contrary to the proof of~\cite[Theorem~4.1]{ViCu2012} and~\cite[Corollary~4.1]{ViCu2012}, the set of trial points at any iteration $k \in \N$ is $\P_k = P^r_k \cup P_k$ instead of~$P_k$.
    Set~$K_* = K_1$.
	Lemma~\ref{prop:behavior_revealing_poll} ensures that $\D$ is almost surely dense in the unit sphere.
    Hence, it suffices to show that~$\cup_{k \in K_*}\P_k$ is dense in $B \cap \B_R(x_*)$, and that $f_* = \lim_{k \in K_*} f(x_k)$ equals $f(x_*)$.
    
    First, let us prove that $\cup_{k \in K_*}\P_k$ is almost surely dense in $B \cap \B_R(x_*)$ and in $B_* \cap \B_R(x_*)$.
    Recall that, for any sets $S_1$ and $S_2$, $\cl(S_1 \cap S_2) = \cl(S_1) \cap \cl(S_2)$ when $\partial S_1 \cap \partial S_2 \subseteq \partial(S_1 \cap S_2)$, and note that $\partial(\int(B)) \cap \partial(\int(\B_R(x_*))) \subseteq \partial(\int(B) \cap\int(\B_R(x_*)))$ and $\partial B \cap \partial \B_R(x_*) \subseteq \partial(B \cap \B_R(x_*))$.
    Note also that Assumption~\ref{hyp:closure_branches} leads to $\cl(\int(B)) = \cl(B)$.
    Then, Lemma~\ref{prop:behavior_revealing_poll} guarantees that $\cup_{k \in K_*} \P_k$ is dense in $\cl(\int(B \cap \B_R(x_*))) = \cl(\int(B) \cap \int(\B_R(x_*)) = \cl(\int(B)) \cap \cl(\int(\B_R(x_*))) = \cl(B) \cap \cl(\B_R(x_*)) = \cl(B \cap \B_R(x_*))$.
    Then $\cup_{k \in K_*} \P_k$ is dense in $B \cap \B_R(x_*)$.
    The second claim holds by a similar argument.
    
	Second, let us prove that $\lim_{k \in K_*}f(x_k) = f(x_*)$.
    The sequence $(f(x_k))_{k \in K_*}$ is decreasing and bounded below, thus it converges, and $f_* \geq f(x_*)$ by lower semicontinuity of $f$ at~$x_*$.
	Moreover, assuming that~$f_* > f(x_*)$ raises a contradiction.
	Indeed it implies that $f(x_k) \geq f(x_*) +\varepsilon$ for some $\varepsilon > 0$ and all $k \in K_*$.
	Then $\inf\{f(x) : x \in \cup_{k \in K_*} \P_k\} \geq f(x_*) + \varepsilon$ since $f(x) \geq f(x_k)$ for all $k \in K_*$ and $x \in \P_k$.
	Thus $\inf\{f(x) : x \in (\cup_{k \in K_*} \P_k) \cap B_*\} \geq f(x_*) + \varepsilon$.
	This contradicts the continuity at $x_*$ of the restriction of $f$ to $B_*$, since $\cup_{k \in K_*} \P_k$ is almost surely dense in $B_* \cap \B_R(x_*)$.
\end{proof}

Following~\cite[Theorem~4.1]{ViCu2012}, its corollary~\cite[Corollary~4.1]{ViCu2012} attempts to provide sufficient conditions to guarantee that a \ddsm generates a refining sequence $(x_k)_{k \in K}$ with a refined point $x_*$ satisfying~$\lim_{k \in K} f(x_k) = f(x_*)$.
Theorem~\ref{th:main_result} recovers this property.
It also guarantees that, for any value of $n_B$, an infinite subset of each branch of $f$ near $x_*$ is evaluated, while~\cite[Corollary~4.1]{ViCu2012} claims this property under the restriction $n_B = 2$.
Note that Theorem~\ref{th:main_result} is established under Assumption 1 which replaces~\cite[Assumption~4.1]{ViCu2012} (precisely the exterior cone property is replaced by the assumption that $\cl(Bi) = \cl(\int(B_i))$ for each $i$).
Finally Theorem~\ref{th:main_result} also ensures the density of the set of refining directions, which is required in~\cite[Theorem~3.2]{ViCu2012} but not ensured by~\cite[Corollary~4.1]{ViCu2012}.
Consequently~\cite[Theorems~3.1 and~3.2]{ViCu2012} remain usable.

%% %%%%%%%%%%%%%%%%%%%%%%%%%%%%%%%%%%%%%%%%%%%%%%%%%
\subsection{Algorithm~\ref{algo:new_mads} handles the counterexample}
%% %%%%%%%%%%%%%%%%%%%%%%%%%%%%%%%%%%%%%%%%%%%%%%%%%

This section shows how Algorithm~\ref{algo:new_mads} handles the counterexample constructed in Section~\ref{sec:counterexample}, thanks to the \revpoll step.
Consider the \ddsm constructed in Section~\ref{sec:counterexample/mads} with the addition of a \revpoll step which generates a single trial point per iteration with the polling radius~$R > 0$ chosen as~$R = 2$ for the ease of presentation.
This instance of Algorithm~\ref{algo:new_mads} satisfies the next Proposition~\ref{prop:behavior_new_mads}, proving that the counterexample constructed in Section~\ref{sec:counterexample} is addressed.
The proof of this proposition can be adapted to remain valid with any polling radius $R > 0$.

\vspace{0.1cm}

\begin{proposition}
	\label{prop:behavior_new_mads}
	The aforementioned instance of Algorithm~\ref{algo:new_mads} almost surely generates a refining subsequence of $(x_k;\alpha_k;\D_k)_{k\in\N}$ with the refined point $x_* = -1$ (the global minimizer of $f$) and $\D = \{-1,1\}$ as the set of refining unitary directions.
\end{proposition}

\begin{proof}
	Denote $I = [-\sqrt{2}-1,0]$ and note that $f(x) \leq 0$ if and only if $x \in I$.
	If there exists $k_0 \in \N$ such that $x_{k_0} \in I$, then all incumbent solutions $x_k$ with $k \geq k_0$ belong to $I$ as well, and thus Algorithm~\ref{algo:new_mads} necessarily converges to $x_* = -1$ since it is the only point satisfying some optimality conditions on $I$.
	Hence it suffices to prove that the index $k_0$ almost surely exists.
	
	At any even iteration $k \in 2\N$, the \revpoll step is executed and the incumbent solution satisfies~$x_k \leq x_0 = \frac{5}{4}$.
	Then, if $x_k > 0$, the \revpoll step has a probability $\frac{2-x_k}{4} > \frac{2-x_0}{4} = \frac{3}{16}$ to generate a trial point~$t_k \in [x_k-2,0] \subset [-2,0] \subset I$ which would imply that $x_{k+1} = t_k \in I$.
	Thus there almost surely exists~$k_0 \in \N$ such that~$x_k \in I$, which concludes the proof.
\end{proof}

%% %%%%%%%%%%%%%%%%%%%%%%%%%%%%%%%%%%%%%%%%%%%%%%%%%
\section{Discussion}
\label{sec:Discussion}
%% %%%%%%%%%%%%%%%%%%%%%%%%%%%%%%%%%%%%%%%%%%%%%%%%%

This work proposes a counterexample to a result announced in~\cite{ViCu2012} on direct search methods applied to discontinuous functions.
It shows in the one-dimensional case that a \ddsm may converge to a point of discontinuity without noticing it, since it evaluates points in only one of the branches of the function near the refined point.
This counterexample may be generalized to any dimension.
The paper also identifies the incorrect step in the proof and proposes an algorithmic modification that recovers the validity of the results presented in~\cite{ViCu2012}.
Future works could improve this modified algorithm in order to save some evaluations of the function.
Another perspective for future works could consist in analyzing additional hypotheses on the objective function and on the domain to preserve the validity of the results without modifying the algorithm.

%% %%%%%%%%%%%%%%%%%%%%%%%%%%%%%%%%%%%%%%%%%%%%%%%%%
\section*{Conflict of interest}
The authors declare that they have no conflict of interest.

%% %%%%%%%%%%%%%%%%%%%%%%%%%%%%%%%%%%%%%%%%%%%%%%%%%
\bibliography{bibliography.bib}

\end{document}